\documentclass{amsart}
\usepackage{amssymb,amsmath,color,hyperref}
\usepackage{color}
\usepackage{tikz}
\usepackage{pgf}
\usepackage{bm}
\usepackage[left=1.2in,right=1.2in,top=1.2in,bottom=1.2in]{geometry}
\usetikzlibrary{calc, positioning, arrows}
\usepackage{stmaryrd}

\newtheorem{theorem}{Theorem}[section]

\newtheorem{lemma}[theorem]{Lemma}
\newtheorem{corollary}[theorem]{Corollary}

\theoremstyle{definition}
\newtheorem{definition}[theorem]{Definition}
\newtheorem{example}[theorem]{Example}

\theoremstyle{remark}
\newtheorem{remark}[theorem]{Remark}
\newtheorem{problem}{Problem}

\def \trop {\operatorname{trop}}
\def \hom {\operatorname{hom}}
\def \Hom {\operatorname{Hom}}
\def \HDE {\operatorname{HDE}}

\newcommand{\NU}{{\mathcal{N}_\mathcal{U}}}

\newcommand{\RR}{\mathbb{R}}

\newcommand{\NN}{\mathbb{N}}
\newcommand \U {\mathcal{U}}

\newcommand{\function}[2]{:#1 \longrightarrow #2}
\newcommand{\of}[1]{\left( #1 \right)}
\newcommand{\wt}{\widetilde}

\date{\today}

\title{On Domination Exponents for Pairs of Graphs}
\thanks{Annie Raymond was partially supported by NSF grants DMS-2054404 and DMS-2338532. Fan Wei was partially supported by NSF grant DMS-2401414. This material is based upon work supported by the National Science Foundation under Grant No.\ DMS-1928930, while the second and fourth authors were in
residence at the Simons Laufer Mathematical Sciences Institute in
Berkeley, California, during the Spring 2025 semester.}

\author{Grigoriy Blekherman}
\address{Georgia Institute of Technology}
\email{greg@math.gatech.edu}

\author{Annie Raymond}
\address{University of Massachusetts Amherst}
\email{annieraymond@umass.edu}
\author{Alexander Razborov}
\address{University of Chicago and Steklov Mathematical Institute}
\email{razborov@uchicago.edu
}
\author{Fan Wei}
\address{Duke University}
\email{fan.wei@duke.edu}

\begin{document}
\maketitle

\begin{abstract}
Understanding graph density profiles is notoriously challenging. Even for
pairs of graphs, complete characterizations are known only in very limited
cases, such as edges versus cliques. This paper explores a relaxation of
the graph density profile problem by examining the homomorphism density
domination exponent $C(H_1, H_2)$. This is the smallest real number $c \geq 0$
such that $t(H_1, T) \geq t(H_2, T)^c$ for all target graphs $T$ (if such a $c$ exists) where
$t(H,T)$ is the homomorphism density from $H$ to $T$. We demonstrate
that infinitely many families of graphs are required to realize $C(H_1,
H_2)$ for all connected graphs $H_1$, $H_2$. We derive the homomorphism density domination exponent for a variety
of graph pairs, including paths and cycles. As a couple of typical
examples, we obtain exact values when $H_1$ is an even cycle and $H_2$
contains a Hamiltonian cycle, and provide asymptotically sharp bounds when
both $H_1$ and $H_2$ are odd cycles.
\end{abstract}

\section{Introduction}
A graph \( H \) consists of a vertex set \( V(H) \) with cardinality $v(H)$ and an edge set \( E(H) \) with cardinality $e(H)$. Throughout this paper, we assume all graphs are simple, meaning they contain no loops or multiple edges. The \emph{number of homomorphisms} from a graph \( H \) to a target graph \( T \), denoted by \( \hom(H, T) \), is the number of vertex mappings from \( V(H) \) to \( V(T) \) that preserve adjacency, i.e., that map every edge of \( H \) to an edge of \( T \). The \emph{homomorphism density} from \( H \) to \( T \), denoted by
\[
t(H, T) := \frac{\hom(H, T)}{v(T)^{v(H)}},
\]
represents the probability that a random mapping from \( V(H) \) to \( V(T) \) yields a graph homomorphism.

Many classical and open problems in extremal combinatorics can be framed in terms of inequalities involving homomorphism densities. Specifically, one often seeks to determine whether a finite linear combination of densities is nonnegative for all graphs \( T \),
\[
\sum c_i t(H_i, T) \geq 0,
\]
 where \( c_i \in \mathbb{R} \) are fixed constants and \( H_i \)
are fixed subgraphs. Sophisticated techniques involving sums of squares and
semidefinite programming have been employed to address these questions (see
e.g.,~\cite{R07,FLS07, LS12}). These methods have significantly advanced our
understanding of asymptotic problems in extremal combinatorics.

A pivotal question in this area was posed by Lov\'asz and Razborov
(see~\cite[Problem 21]{LLP} and \cite[Question 2]{R07}, respectively): is
every valid linear inequality involving homomorphism densities expressible as
a sum of squares?  Given that the product of two densities satisfies
\[
t(G, T)t(H, T) = t(GH, T),
\]
where \( GH \) is the disjoint union of graphs \( G \) and \( H\), this
question is equivalent to asking the same question for {\em polynomial}
inequalities in homomorphism densities.

Hatami and Norin~\cite{HN} provided a strong negative answer by demonstrating
that determining whether a polynomial inequality involving homomorphism
densities holds for all graphs is undecidable. Subsequent work has extended
this undecidability result~\cite{BRW, MW}, highlighting the intrinsic
difficulty of such problems.

A slightly more general framework is based on the notion of graph profile:
given a finite set of distinct graphs \( \mathcal{U} = \{H_1, H_2, \dots, H_k\}
\), the associated \emph{graph density profile} \( \mathcal{D}_\mathcal{U} \) is
defined as the closure in $\RR^k$ of the set of \( k \)-tuples
\[
(t(H_1, T), t(H_2, T), \dots, t(H_k, T)),
\]
where \( T \) ranges over all possible graphs (or graphons). This set is
connected but it is known not to be always convex and the linear inequalities
problem is equivalent, by simple duality, to describing its convex hull.

Understanding graph profiles is notoriously challenging. For example, even
when \( |\mathcal{U}| = 2 \), complete characterizations are known only in a
limited number of cases. A cornerstone result in this direction is the
resolution when \( \mathcal{U} = \{K_r, K_2\} \), where one graph is a clique
and the other is an edge~\cite{triangle, Niki, Reiher}. For comparison, the
convex hull of this profile was described in a much earlier paper \cite{Bol}
using a significantly simpler (and very elegant) method.

Another natural relaxation of the general ``profile problem'' is to look at
the convex hull of its {\em logarithm} (defined coordinatewise).
Equivalently, one seeks to describe the set of all tuples $(c_1,\ldots,c_k)\in \mathbb{R}^k$
for which the inequality
\begin{equation}
    t(H_1,T)^{c_1}t(H_2,T)^{c_2}\ldots t(H_k, T)^{c_k} \geq 1 \label{eq:bound}
\end{equation}
holds for all target graphs $T$. We call such inequalities \emph{pure
binomial inequalities}. For various profiles, notably comprised of paths,
cycles and stars, these questions were studied by many authors, see e.g.,
\cite{Alo,Sid94,Stoner}. Remarkably, we do not know yet whether the analogue
of \cite{HN} is true in this context, that is if the set of all pure binomial
inequalities is decidable or not.

As has been shown in \cite{BRST}, the convex hull of the logarithm of the
coordinates is equal to the \emph{tropicalization} of
$\mathcal{D}_\mathcal{U}$, a concept in algebraic geometry \cite{tropalg},
which is defined to be $ \textup{trop}(\mathcal{D}_{\mathcal{U}}) := \lim_{t
\rightarrow 0} \log_{\frac{1}{t}} (\mathcal{D}_{\mathcal{U}})$ where
$\log_b(\mathcal{D}_{\mathcal{U}})$ for some base $b$ denotes the image of
$\mathcal{D}_{\mathcal{U}} \cap \RR^k_{>0}$ under the map $\mathbf{v} \mapsto
(\log_b v_1, \ldots, \log_b v_k)$. The tropicalization captures all valid pure
binomial inequalities on $\mathcal{D}_\mathcal{U}$; more specifically, the
dual of the cone $\textup{trop}(\mathcal{D}_\mathcal{U})$ is the set of all
tuples $(c_1,\ldots,c_k)\in \mathbb{R}^k$ for which the inequality
(\ref{eq:bound}) holds for all target graphs $T$. Previous works have used tropicalization in advancing the understanding of graph density profiles for cliques, paths and stars, and also graph homomorphism number profiles for even cycles and odd cycles separately ~\cite{BRST, BR}.

\medskip
Our paper almost exclusively deals with the case $k=2$. In this case,
we are trying to determine the sharpest polynomial inequality
\[
t(H_1,T) \geq t(H_2,T)^c
\]
that holds for all graphs \( T \). This problem leads to the concept of the graph density
{\it domination exponent}, which we formally define later in
Definition~\ref{def}. Intuitively, given two graphs \( H_1 \) and \( H_2 \),
the domination exponent \( C(H_1,H_2) \) is the smallest real
number \( c \) such that $ t(H_1,T) \geq t(H_2,T)^c $ for all target graphs \(
T \).

The notion of domination exponent has been studied implicitly in various
contexts for a long time, but was formally defined only more recently in
~\cite{KR} for homomorphism numbers, and in \cite{Stoner} for homomorphism densities\footnote{In \cite{Stoner}, the notation $\rho(H_2,
H_1)$ is used. We believe that our notation is slightly more natural since, as will
become clear later, this concept is more relevant to homomorphisms from $H_1$
to $H_2$ than the other way around.}. It has been recognized as a unifying concept for several
long-standing problems in extremal combinatorics. For instance, Sidorenko's
famous conjecture~\cite{Sid93}, which posits that
\[
t(H,T) \geq t(K_2,T)^{e(H)}
\]
for any bipartite graph \( H \), can be equivalently stated as asserting that
\( C(H, K_2) =e(H).\) Despite decades of effort, Sidorenko's conjecture
remains open.

In this paper, we explore several foundational questions about density domination
exponents and their applications. We also give several new computations for
specific graph pairs, notably cycles and paths. We also present open problems
and conjectures that we believe will stimulate further research in this area.

\subsection{Our Results}

The first question we address is which classes of target graphs $G$ are
sufficient to realize the values of $C(H_1, H_2)$. For example, Sidorenko's
conjecture states that if $H_1$ is a bipartite graph and $H_2$ is the edge
$K_2$, then $C(H_1, H_2)$ is achieved by the limit of the Erd\H{o}s-R\'enyi
graph $G(n,p)$ for any fixed $0 < p < 1$.

We show that in order to realize the values $C(H_1, H_2)$ for all connected
graphs $(H_1, H_2)$, we need infinitely many families of graphs $\{T_n\}$ so
that
\[
\lim_{n\to\infty}\frac{\log t(H_1, T_n)}{\log t(H_2, T_n)} = C(H_1,H_2).
\]
This remains true even when $H_1, H_2$ are confined to be even cycles. This
shows that to realize domination exponents for any pair of even cycles, we
need infinitely many constructions, which resolves \cite[Question
6.4]{Stoner} and is presented in
Theorem~\ref{thm:infinitelymanynotconnected}.

Secondly, we explore general methods  to determine the value of $C(H_1,
H_2)$, and explicitly compute it for some classes of graphs.

To that end, in Section~\ref{sec:misc}, we establish general facts about
graph homomorphisms and the homomorphism density domination exponent. We use
them to compute $C(H_1, H_2)$ under various general assumptions (see Theorem
\ref{thm:misc}). Additionally, in Theorem~\ref{lem:K4-eK3}, we determine that
$C(K_4 - e, K_3) = 2$, matching a lower bound proved by Tao \cite[Exercise
16.21]{LL12} \footnote{Tao proved a stronger result that we will cite in the
proof of Theorem \ref{lem:K4-eK3}.}.

The most technical part of our paper, where we also introduce several new
methods that may be of independent interest, is in Sections
\ref{sec:paths} and \ref{sec:cycles}. In Section~\ref{sec:paths}, we focus on the
cases where both $H_1$ and $H_2$ are paths. Previously, Stoner \cite{Stoner}
computed $C(P_k, P_\ell)$ for all $k$ and $\ell$, except when $k < \ell$ are
both odd and $k+1$ does not divide $\ell+1$ (where $P_m$ denotes the path with $m$ edges). We complete this missing case
and thus completely determine $C(P_k, P_\ell)$ for all $k, \ell$.
Specifically, using results from \cite{BR}, we show that
\[
C(P_k,P_\ell) = \frac{k+\ell-r}{(a+1)\ell},
\]
where $k < \ell$ are both odd and $\ell = a(k+1)+r,\ 0\leq r\leq k$.

In Section~\ref{sec:cycles}, we study the homomorphism density domination
exponent for cycles. The values of $C(C_k, C_\ell)$ are already known when $k
< \ell$ and $k$ is even \cite[Proposition 4.3]{Stoner}. We extend this result
by completely determining the domination exponent except in the case where
{\em both} are odd cycles and $k < \ell$. For this remaining case, we
establish both upper and lower bounds that are strong enough to
asymptotically determine $C(C_k, C_\ell)$, addressing a question posed by
Stoner \cite{Stoner}.

Finally, as a corollary to our results on cycles, we compute the tropicalization of the density profile for edges and even cycles. This allows us to efficiently compute $C(G,H)$ for any graphs $G$ and $H$ that are disjoint unions of edges and even cycles. These findings are presented in Section~\ref{sec:cycles}.

We conclude by presenting several open questions and conjectures in
Section~\ref{sec:open} that arise naturally from our work.

\section{Preliminaries and some general results}\label{sec:prelims}

In this section, we start with some general facts about graph homomorphisms
and homomorphism density domination exponents, and then show that any set of
optimizers for homomorphism domination exponents must be infinite, thus
solving \cite[Question 6.4]{Stoner}.

\subsection{General facts about homomorphisms and homomorphism domination exponents}

We begin right away with our main definition.

\begin{definition}\label{def}
For graphs $G$ and $H$, the \emph{homomorphism density domination exponent}
$C(G,H)$ is defined by\footnote{Since the set in question is an intersection
of closed sets and hence is itself closed, the minimum exists.}
$$C(G,H)= \min\{c\in \mathbb{R}\ |\ t(G,T) \geq t(H,T)^c \textup{ for all graphs } T\}.$$
Note that we can also express $C(G,H)$ as
$$C(G,H)= \sup_{0<t(H,T)<1} \frac{\log(t(G,T))}{\log(t(H,T))}.$$
\end{definition}

Thus any graph $T$ (or any sequence of graphs) yields a lower bound on
$C(G,H)$, and proving that $t(G,T)\geq t(H,T)^c$ is a valid inequality yields
an upper bound on $C(G,H)$, namely $c$.

\bigskip Next we proceed to describe several useful constructions on graphs
and their behavior with respect to homomorphism densities.

For a graph $G$ with $V(G)=\{v_1, \ldots, v_k\}$, we let $G'(a_1, \ldots, a_k)$ be the \emph{blow-up} of $G$ where

\begin{equation} \label{def:blowup}
\begin{gathered}
V(G'(a_1, \ldots, a_k)) =\{(v_i,j)|i\in [k], j\in [a_i]\},\\
E(G'(a_1, \ldots, a_k))  =\{\{(v_{i_1}, j_1),(v_{i_2},j_2)\}\in V(G'(a_1, \ldots, a_k)) \times V(G'(a_1, \ldots, a_k)) | \{v_{i_1}, v_{i_2}\}\in E(G)\}.
\end{gathered}
\end{equation}

We denote by $GH$ the disjoint union of graphs $G$ and $H$ and by $G\times H$ the categorical (tensor) product of $G$ and $H$, namely the graph with vertex set $V(G) \times V(H)$ and edge set $\{\{(u_1,v_1),(u_2,v_2)\} |$ $\{u_1,u_2\} \in E(G) \textup{ and } \{v_1,v_2\} \in E(H)\}$. One can easily check that $t(G,T)\cdot t(H,T) = t(GH, T)$ and that $t(G,T_1) \cdot t(G,T_2)= t(G, T_1\times T_2)$.

It will be also sometimes convenient for our purposes to use some of the
language related to the gluing algebra of graphs; we refer the reader to
Lov\'{a}sz~\cite{LL12} for a broader exposition.

A graph is {\em partially labeled} if a subset of its vertices is labeled
with elements of $\mathbb{N} := \{1,2,3,\ldots\}$ such that no vertex
receives more than one label. If no vertices of $H$ are labeled, then $H$ is
{\em unlabeled}.

Labeled graphs $H_1$ and $H_2$ can be multiplied ({\em glued}) as follows.
Given $H_1$ and $H_2$, we form the new labeled graph $H_1H_2$ by gluing
together the vertices in the two graphs with the same label, and keeping only
one copy of any edge that may have doubled in the process. (Note that if
$H_1$ and $H_2$ are unlabeled, $H_1H_2$ is just the disjoint union of the
graphs as above.)

For a labeled graph $H$, the (ordinary) graph $\llbracket H \rrbracket$ is
obtained from $H$ by removing all labels.

\bigskip
We now show a necessary and sufficient condition for $C(G,H)$ to exist. To do
so, we first need the following lemma.

\begin{lemma}\label{lem:blowupupperbound}
Let $H$ be a graph on $k$ vertices and $a_1, \ldots, a_k\in \NN^+$. Then
$t(H'(a_1, \ldots, a_k),T) \geq t(H,T)^{a_1a_2\ldots a_k}$ for every graph
$T$, where $H'$ is the blowup graph defined in (\ref{def:blowup}).
\end{lemma}

\begin{proof}
First assume that $a_1>1$ and $a_2=\ldots=a_k=1$. Let $G$ be the graph $H$
with all vertices labeled except for $v_1$. Then $t(H,T)=t(\llbracket
G\rrbracket,T)$ and $t(H'(a_1, 1, \ldots, 1),T)=t(\llbracket
G^{a_1}\rrbracket,T)$. By H\" older's inequality, we have that $t(H'(a_1, 1,
\ldots, 1),T) \geq t(H,T)^{a_1}$.

Now the statement follows by sequentially blowing up each vertex of $H$ and
using the previous argument.
\end{proof}

The next theorem provides a characterization of when $C(G,H)$ exists and a
(crude) upper bound on it when it does.
\begin{theorem}
    \label{lem:existence}
We have that $C(G,H)$ exists if and only if $t(G,H)>0$, in which case
$$
C(G,H)\leq \begin{cases}\of{\frac{v(G)}{v(H)}}^{v(H)} & \text{\rm if}\ v(H)\leq
\frac{v(G)}{e},\\
 e^{v(G)/e} & \text{\rm if}\  v(H)\geq
\frac{v(G)}{e}  \end{cases}
$$
{\rm (where }$e=2.7182\ldots${\rm )}.
\end{theorem}

\begin{proof}
$\Rightarrow$: Suppose that $t(G,H)=0$. If we let $T=H$, then there does not
exist a constant $c\in \RR$ such that $t(G,T)\geq t(H,T)^c$ since $t(H,H)>0$.

$\Leftarrow$: If $t(G,H)>0$, let us fix a homomorphism
$\varphi\function{V(G)}{V(H)}$. Assume without loss of generality that $V(H)=[n]$ and
$\text{im}(\varphi)=[r]$ for some $r\leq n$. Let $a_i=|\varphi^{-1}(i)|\geq 1$ so
that $\sum_{i=1}^ra_i=v(G)$. Then $G$ is a subgraph of
the blowup $H'(a_1,\ldots,a_r,1,\ldots,1)$ so by Lemma \ref{lem:basicproperties}(4)
below, Lemma \ref{lem:blowupupperbound} and the AM-GM inequality we conclude
$$
t(G,T)\geq t(H'(a_1,\ldots,a_r,1,\ldots,1),T) \geq t(H,T)^{\prod_{i=1}^ra_i} \geq t(H,T)^{\of{\frac{v(G)}{r}}^r}.
$$
It only remains to note that the function $\of{\frac{v(G)}{r}}^r$ is
increasing for $r\leq \frac{v(G)}{e}$ and decreasing for $r\geq
\frac{v(G)}{e}$.
\end{proof}

As in \cite{KR} and \cite{Stoner}, we now state a few basic properties of the
homomorphism density domination exponent that directly follow from
definitions.

\begin{lemma}[Basic properties]\label{lem:basicproperties}\leavevmode
\begin{enumerate}
\item Let $G\succeq H$ if $C(G,H)\geq 1$.Then $\succeq$ is a partial order on graphs.
\item $C(G^a,H^b) = \frac{a}{b}C(G,H)$.
\item $C(F,H) \leq C(F,G) \cdot C(G,H)$.
\item If $G$ is a subgraph of $H$, then $C(G,H)\leq 1$.
\end{enumerate}
\end{lemma}

Next, we present simple general lower bounds for $C(G,H)$.

\begin{lemma}[\protect{\cite[Proposition 2.6]{Stoner}}]\label{lem:simple}

If $G$ and $H$ are connected graphs, then
\begin{equation}
  \label{eq:simplebound}
C(G, H) \geq \max\left( \frac{e(G)}{e(H)}, \ \ \ \frac{v(G)-1}{v(H)-1}, \ \ \ \frac{v(G)}{v(H)}\right).
\end{equation}
\end{lemma}

\begin{proof}
These lower bounds can be obtained from $G(n,1/2)$ (for $\frac{e(G)}{e(H)}$),
taking two disjoint copies of $K_n$ (for $\frac{v(G)-1}{v(H)-1}$) and
dilating $K_n$ with $n$ isolated vertices (for $\frac{v(G)}{v(H)}$).
\end{proof}

Let us say that a pair $(G, H)$ is {\it generalized Sidorenko} if $ C(G,H) =
e(G)/e(H)$. Many pairs of the form $(P_k,P_\ell)$ (\cite[Theorem
4.1]{Stoner}) and $(C_{2k},C_\ell)$ (\cite[Proposition 4.3]{Stoner}) are
generalized Sidorenko so this notion is not meaningless. On the other hand,
as a corollary of the above lemma, any pair $(G,H)$ such that
$\frac{e(G)}{e(H)} < \frac{v(G)-1}{v(H)-1}$, is not generalized Sidorenko. In
particular, if $G$ is a tree then $C(G,K_3)\geq \frac{e(G)}{2}$ (=
$C(G,K_2)$) so $(G,K_3)$ is very far from being generalized Sidorenko.

\subsection{What about the optimizers?}
Sidorenko's conjecture posits that if $H_1$ is a bipartite graph and $H_2$ is
the edge then $C(H_1, H_2)$ is achieved by the limit of the Erd\H{o}s-R\'enyi
graph $G(n,p)$ for any fixed $0 < p < 1$ or, in other words, by any
(non-trivial) constant graphon $W$. All three inequalities in Lemma
\ref{lem:simple} also employ one fixed graphon each. Hence a natural question
\cite[Question 6.4]{Stoner} is this: is the minimal set of optimizers needed
to realize all possible values of $C(H_1, H_2)$ finite? In the following
theorem, we answer it in the negative even when $H_1$ and $H_2$ are
restricted to be even cycles.

\begin{theorem}\label{thm:infinitelymanynotconnected}
Fix finitely many graph sequences $\{T_n^{(1)}\}, \{T_n^{(2)}\}, \dots,
\{T_n^{(N)}\}$, where for each $1 \leq i \leq N$, $\{T_n^{(i)}\} = T_1^{(i)},
T_2^{(i)}, \dots$ is a sequence of graphs or graphons. Then there exist even
cycles $G,H$ such that $\frac{\log(t(G,T_n^{(i)}))}{\log(t(H,T_n^{(i)}))}$
does not converge to $C(G,H)$ as $n\rightarrow \infty$ for any $1 \leq i \leq
N$.
\end{theorem}
\begin{proof}
Suppose that we know $C(F,G)$ and $C(G,H)$ and that the optimum is achieved
via the same sequence of graphs in both cases. Then we also know that
$C(F,H)=C(F,G) C(G,H)$. Therefore, if we have a triple of graphs for which
Lemma \ref{lem:basicproperties}(3) is  not tight, then we know that $C(F,G)$
and $C(G,H)$ require different constructions.

In Theorem~\ref{thm:even_cycles}, we will show that $C(C_{2k},C_{2j})=
\frac{4k(k-1)}{4kj-2k-2j}$ if $k\geq j$. So for $k > j > i$ we have that

\begin{align*}
C(C_{2k},C_{2i}) &= C(C_{2k},C_{2j})C(C_{2j},C_{2i})\\
\Leftrightarrow \frac{4k(k-1)}{4ki-2k-2i} &= \frac{4k(k-1)}{4kj-2k-2j}\cdot \frac{4j(j-1)}{4ji-2j-2i}\\
\Leftrightarrow  (4kj-2k-2j)\cdot(4ji-2j-2i)&= (4ki-2k-2i) \cdot 4j(j-1)\\
\Leftrightarrow 0 &= (k-j)(j-i).
\end{align*}

Therefore, we see that $C(C_{2k},C_{2i}) \neq C(C_{2k},C_{2j})C(C_{2j},C_{2i})$ for any $k>j>i$.

We claim that we need infinitely many constructions to realize
$C(C_{2j},C_{2i})$ where $j>i$ in general. Suppose not and that a finite
number of constructions $N$ suffices. Let $K_m$ have vertices $[m]$ and color
the edge $\{i,j\}\in E(K_m)$ for $i<j$ with any of the colors $c\in [N]$ for
which construction $c$ is optimal for $C(C_{2j},C_{2i})$.  For $m$ large
enough, Ramsey theory tells us that we will find a monochromatic triangle in
$K_m$, which leads to a contradiction since we know $C(C_{2k},C_{2i}) \neq
C(C_{2k},C_{2j})C(C_{2j},C_{2i})$ for any $k>j>i$.
\end{proof}

\section{Homomorphism density domination exponent for paths}\label{sec:paths}

Recall that we denote by $P_m$ the path on $m$ edges. In \cite{Stoner}, Stoner computed
$C(P_k,P_\ell)$ for all $k$ and $\ell$ except for the case when $k<\ell$ are
both odd and $k+1$ does not divide $\ell+1$. In this section, we calculate
$C(P_k,P_\ell)$ for the remaining open case.

\begin{lemma}
Let $k, l, m$ be positive integers such that $m \leq k$. Then

$$C(P_{2k-1},P_{2kl+2m-1}) = \frac{2kl+2k-1}{(l+1)(2kl+2m-1)}.$$
{\rm (}When $m=k$, this was already proved in \cite[Theorem 4.1]{Stoner}.{\rm
)}
\end{lemma}

\begin{proof}
For the upper bound, \cite[Theorem 4.1]{Stoner} shows that
$$t(P_{2k-1},T)^{l+1} \geq t(P_{2k(l+1)-1},T).$$ Moreover, the Erd\H{o}s-Simonovits inequality \cite{saglam, BR2} states that
$$t(P_{2k(l+1)-1},T)^{2kl+2m-1} \geq t(P_{2kl+2m-1},T)^{2k(l+1)-1}.$$ Together, this yields that
$$t(P_{2k-1},T)^{(l+1)(2kl+2m-1)} \geq t(P_{2kl+2m-1},T)^{2k(l+1)-1}.$$

For the lower bound, we use the path blow-up construction $T_n$ in
\cite[Definition 5.7]{BR}, which we explain here for completeness. Let
$b=2l+1$, $s=2l$, and $\mathbf{d}=(d_0, d_1, \ldots d_{kl})$ where $d_i=1$ if
$i=0 \mod k$ and $d_i=0$ otherwise. Let $T_n$ be the blow-up of $P_{2kl+1}$
with weights $p:=p_{b, s, \mathbf{d}}:V(P_{2kl+1})\cup E(P_{2kl+1}) \rightarrow \mathbb{R}$
such that
 \begin{itemize}
\item  $p(\{0\})=b$,
\item $p(\{1\})=b-s=d_0$,
\item $p(\{2u+1\})=d_0+2\sum_{v=1}^u d_v$ for all $1\leq u \leq \lfloor \frac{kl-1}{2} \rfloor$,
\item $p(\{2u\})=s-d_0-2\sum_{v=1}^{u-1} d_v$ for all $1 \leq u\leq \lfloor\frac{kl}{2}\rfloor$,
\item $p(\{2kl+1-u\})=p(u)$ for all $0\leq u \leq kl$,
\item $p(\{2u, 2u+1\})=s+d_u$ for all $0\leq u \leq \lfloor \frac{kl-1}{2}\rfloor$,
\item $p(\{2u-1, 2u\})=s$ for all $1\leq u \leq \lfloor \frac{kl-1}{2} \rfloor $,
\item $p(\{kl,kl+1\})=s+d_{\frac{kl}{2}}$ if $kl$ is even,
\item $p(\{kl,kl+1\})=p(kl)+p(kl+1)$ if $kl$ is odd,
\item $p(\{2kl-u,2kl+1-u\})=p(\{u,u+1\})$ for all $0 \leq u \leq kl-1$,
\end{itemize} i.e., a vertex $v$ of $P_{kl+1}$ becomes an independent set of
size $n^{p(\{v\})}$, and there are $n^{p(\{v,v+1\})}$ edges between the
independent sets corresponding to vertices $v$ and $v+1$. Then by
\cite[Theorem 5.9]{BR}, we have $$\hom(P_{2k-1}; T_n)= O(n^{k\cdot (2l) +
1})$$ as $n\rightarrow \infty$ and thus, since $T_n$ has $O(n^b)$ vertices,
$$t(P_{2k-1}, T_n)= O\left(\frac{n^{k\cdot (2l) + 1}}{n^{(2l+1)2k}}\right)=
O(n^{-2kl-2k+1}).$$ We also have
$$\hom(P_{2kl+2m-1},T_n)=O(n^{(kl+m)\cdot(2l)+l+1})$$ and thus
$$t(P_{2kl+2m-1},T_n)=O\left(\frac{n^{(kl+m)\cdot(2l)+l+1}}{n^{(2l+1)(2kl+2m)}}\right)=O(n^{(l+1)(-2(kl+m)+1)}).$$

\end{proof}

\begin{example}
Let $k=3, l=2, m=1$. Then $\mathbf{d}=(1,0,0,1,0,0,1)$, and $T_n$ is the blow-up of the following graph

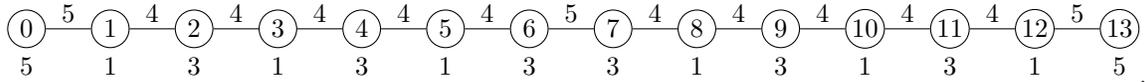
\begin{center}
\begin{tikzpicture}
\node[circle, draw, inner sep=1pt, minimum size=0.5cm, label=below:5] (0) {0};
\node[circle, draw, inner sep=1pt, minimum size=0.5cm, label=below:1] (1) [right = 0.6cm of 0] {1};
\node[circle, draw, inner sep=1pt, minimum size=0.5cm, label=below:3] (2) [right = 0.6cm of 1] {2};
\node[circle, draw, inner sep=1pt, minimum size=0.5cm, label=below:1] (3) [right = 0.6cm of 2] {3};
\node[circle, draw, inner sep=1pt, minimum size=0.5cm, label=below:3] (4) [right = 0.6cm of 3] {4};
\node[circle, draw, inner sep=1pt, minimum size=0.5cm, label=below:1] (5) [right = 0.6cm of 4] {5};
\node[circle, draw, inner sep=1pt, minimum size=0.5cm, label=below:3] (6) [right = 0.6cm of 5] {6};
\node[circle, draw, inner sep=1pt, minimum size=0.5cm, label=below:3] (7) [right = 0.6cm of 6] {7};
\node[circle, draw, inner sep=1pt, minimum size=0.5cm, label=below:1] (8) [right = 0.6cm of 7] {8};
\node[circle, draw, inner sep=1pt, minimum size=0.5cm, label=below:3] (9) [right = 0.6cm of 8] {9};
\node[circle, draw, inner sep=1pt, minimum size=0.5cm, label=below:1] (10) [right = 0.6cm of 9] {10};
\node[circle, draw, inner sep=1pt, minimum size=0.5cm, label=below:3] (11) [right = 0.6cm of 10] {11};
\node[circle, draw, inner sep=1pt, minimum size=0.5cm, label=below:1] (12) [right = 0.6cm of 11] {12};
\node[circle, draw, inner sep=1pt, minimum size=0.5cm, label=below:5] (13) [right = 0.6cm of 12] {13};
\draw (0)--(1) node [midway, inner sep=1pt, above=0.05cm, fill=white] {5};
\draw (1)--(2) node [midway, inner sep=1pt,above=0.05cm, fill=white] {4};
\draw (2)--(3) node [midway, inner sep=1pt,above=0.05cm, fill=white] {4};
\draw (3)--(4) node [midway, inner sep=1pt,above=0.05cm, fill=white] {4};
\draw (4)--(5) node [midway, inner sep=1pt,above=0.05cm, fill=white] {4};
\draw (5)--(6) node [midway, inner sep=1pt,above=0.05cm, fill=white] {4};
\draw (6)--(7) node [midway, inner sep=1pt,above=0.05cm, fill=white] {5};
\draw (7)--(8) node [midway, inner sep=1pt,above=0.05cm, fill=white] {4};
\draw (8)--(9) node [midway, inner sep=1pt,above=0.05cm, fill=white] {4};
\draw (9)--(10) node [midway, inner sep=1pt,above=0.05cm, fill=white] {4};
\draw (10)--(11) node [midway, inner sep=1pt,above=0.05cm, fill=white] {4};
\draw (11)--(12) node [midway, inner sep=1pt,above=0.05cm, fill=white] {4};
\draw (12)--(13) node [midway, inner sep=1pt,above=0.05cm, fill=white] {5};
\end{tikzpicture}.
\end{center}

In this picture, the circled numbers from $0$ to $13$ are the labels of the vertices. The integer below vertex $i$ for $0 \leq i \leq 13$ is $p(\{i\})$, i.e., in the blow-up construction, each vertex $i$ becomes an independent set of size $n^{p(\{i\})}$.  The integer on top of each edge between $i$ and $i+1$
 denotes the value of $p(\{i, i+1\})$. In other words, there are $n^{p(\{i, i+1\})}$ edges between the independent sets corresponding to vertices $i$ and $i+1$.

There are $O(n^{13})$ homomorphisms from $P_5$ to $T_n$, and $O(n^{31})$ homomorphisms from $P_{13}$ to $T_n$. Since $T_n$ has $O(n^5)$ vertices, we get that $\log(t(P_5,T_n))\rightarrow -17$ and $\log(t(P_{13},T_n))\rightarrow -39$ as $n\rightarrow \infty$, which yields that $C(P_5,P_{13})\geq \frac{17}{39}$.
\end{example}

With Stoner's results, this yields the following theorem.

\begin{theorem}\label{thm:Cpaths}
Let $k,\ell\geq 1$. Then

$$C(P_k, P_\ell)=\left\{
\begin{array}{ll}
\frac{k}{\ell} & \textup{if } k>\ell \textup{ and either } k \textup{ is even or both } k,\ell \textup{ are odd,}\\
\frac{k+1}{\ell} & \textup{if } k \textup{ is odd and } \ell \textup{ is even,}\\
\frac{k+1}{\ell+1} & \textup{if } k<\ell \textup{ and } k \textup{ is even, and}\\
\frac{k+\ell-r}{(a+1)\ell}& \textup{if } k<\ell \textup{ and } k,\ell \textup{ are odd}
\end{array} \right.$$
where, in the last line, $a=\lfloor \frac{\ell}{k+1}\rfloor$ and $r$ is the remainder of that division.
\end{theorem}


 \section{Homomorphism density domination exponent for cycles}\label{sec:cycles}
Let $C_k$ be the cycle on $k$ vertices and, for the sake of uniformity, let
$C_2=K_2$ be the graph consisting of a single edge. In \cite[Proposition
4.3]{Stoner}, Stoner computed $C(C_k, C_\ell)$ when $k<\ell$ and $k$ is even.
In this case, one could also retrieve $C(C_k, C_\ell)$ from the result of
Lov\'{a}sz on concavity of even cycles \cite[Lemma 3.9]{LL11}. We now
study the remaining cases.

By Theorem \ref{lem:existence}, $C(C_k,C_\ell)$ exists if and only if
$t(C_k,C_\ell)\neq 0$. Therefore, if $C_k$ is an odd cycle and $C_\ell$ is an
even cycle, then $C(C_k, C_\ell)$ does not exist. Moreover, if both $C_k,
C_\ell$ are odd cycles and $C_k$ is shorter, then $C(C_k, C_\ell)$ does
not exist either.

The only remaining cases are (1) $C_k$ is an even cycle, and (2) $C_k$ is an odd cycle, and $C_\ell$ is a shorter odd cycle. We solve case (1) completely, and for case (2), we prove new upper and lower bounds that determine $C(C_k,C_l)$ asymptotically.

\subsection{Characterization of $C(C_{2k}, C_\ell)$}

As mentioned above, when $2k\leq\ell$, the value $C(C_{2k},
C_\ell)=\frac{2k}{\ell}$ was already computed in \cite{Stoner}.

In the missing case $2k > \ell$, we have the following theorem.

 \begin{theorem} \label{thm:even_cycles}
     Let $k, \ell$ be positive integers with $k,\ell\geq 2$ and $2k\geq \ell$.
     Then  \[C(C_{2k}, C_{\ell}) = \frac{4k(k-1)}{2k\ell-2k-\ell}.\]
 \end{theorem}

\begin{proof}
For the upper bound, note that we have the following inequality
\begin{equation} \label{eq:concavity}
t(C_2,T)^{\frac{2k-\ell}{\ell -2}} t(C_{2k},T) \geq t(C_{\ell},T)^{\frac{2k-2}{\ell-2}}.
\end{equation}
One way of proving this is to notice that if $M$ is the adjacency matrix of
$T$ and $v=\of{\frac 1n\lambda_1, \ldots, \frac1n\lambda_n}$ is the vector of
its normalized eigenvalues, then $t(C_r,T)\leq ||v||_r^r$, with equality when
$r$ is even. Plugging this spectral interpretation into \eqref{eq:concavity},
we get a special case of H\"older's inequality.

Together with Sidorenko's inequality $t(C_{2k},T) \geq t(C_2,T)^{2k}$, we now
have
 \[
 t(C_{2k},T)\geq t(C_{\ell},T)^{ \frac{4k(k-1)}{2k\ell-2k-\ell}}.
 \]

 We are left to find a construction that yields the matching lower bound.

For the rest of the proof fix $k,\ell$ such that $k,\ell\geq 2$ and
$2k\geq\ell$, and consider a finite projective plane of order $n\to\infty$,
say $n=p^2+p+1$, where $p$ is a large prime number. Pick any $L = \lfloor
n^{\alpha}\rfloor$ lines and call them {\em red} where
\begin{equation} \label{eq:alpha_def}
\alpha = \frac{k}{2k-1}.
\end{equation}
Form the graph $T_n$ whose vertices are the points of the projective plane,
with two points joined by an edge if and only if the line passing through
them is red. In other words, $T_n$ is the union of (edge-disjoint) cliques
induced by red lines.

To lower bound $\hom(C_\ell, T_n)$, notice that there are $L$ red lines, and each red line yields a clique of size $p+1$ in $T_n$, so each red line contributes $\Theta(p^{\ell})$ cycles of length $\ell$. Thus
$$\operatorname{hom}(C_{\ell},T_n)\geq \Omega(L p^{\ell}) = \Omega(n^{\alpha + \ell/2}).$$
This implies that
\begin{equation} \label{eq:Cl_Tn}
t(C_{\ell},T_n) \geq \Omega(n^{\alpha + \ell/2 - \ell}) = \Omega(n^{\alpha - \ell/2}).
\end{equation}

We now upper bound $\hom(C_{2k},T_n)$. Let $\mathcal L$ be the set of all red
lines and let $v_1,\dots,v_{2k}$ be the vertices of $C_{2k}$. A homomorphism
of $C_{2k}$ into $T_n$ defines a map $\varphi\function{[2k]}{\mathcal L}$, where
$\varphi(i)$ is the red line passing through $v_i, v_{i+1}$ (the subscript
addition is modulo $2k$). There are at most $O(L^s)$ (recall that $k$ is an
absolute constant) such mappings with $|\text{im}(\varphi)|=s$. This is because there are $\binom{L}{s}\leq L^{s}$ ways to choose $\text{im}(\phi)$, and then $O(1)$ ways to construct $\phi$ with this image.

We now bound the number of $C_{2k}$ corresponding to a fixed $\varphi$. This
mapping splits $C_{2k}$ in at least $s$ arcs corresponding to different red
lines according to the value of $\varphi$. There are at least $s$ points that
are border points between different arcs, unless $s=1$. The vertices in $T_n$
corresponding to these border points are uniquely determined by $\varphi$
because every two different lines uniquely determine a point. For every
internal point (there are at most $2k-s$ of them, again unless $s=1$), we have
at most $p$ choices since each line has $p$ points. Thus, any fixed $\varphi$
with $|\text{im}(\varphi)|=s$ contributes to $\operatorname{hom}(C_{2k},T_n)$ at
most $p^{2k-s}$ when $s\geq 2$ and at most $p^{2k}$ when $s=1$.

Altogether, this gives us at most
\[
\sum_{s = 2}^{2k} O(L^sp^{2k-s}) \leq  O\left( \max_{2 \leq s \leq 2k}  n^{\frac{k s }{2k-1}} n^{\frac{2k-s}{2}} \right) \leq  O( n^{\frac{2k^2}{2k-1}})
\]
choices when $s\geq 2$ and also at most
\[ O(L p^{2k}) \leq O (n^{\frac{k}{2k-1} + k} ) \leq  O(n^{\frac{2k^2}{2k-1}})\]
choices when $s=1$. Thus \[\hom(C_{2k},T_n) \leq  O (n^{\frac{2k^2}{2k-1}})
\] which implies
\[
t(C_{2k},T_n) \leq O(n^{\frac{2k^2}{2k-1} - 2k}).
\]
By plugging in the value of $\alpha$, the desired bound holds.
\end{proof}

Observing that \eqref{eq:Cl_Tn} actually improves to $t(K_\ell,T_n)\geq
\Omega(n^{\alpha-\ell/2})$, we immediately get the following generalization.
\begin{theorem}
 Let $k, \ell$ be positive integers with $k,\ell\geq 2$ and $2k\geq \ell$, and let
 $H$ be an arbitrary graph on $\ell$ vertices with a Hamiltonian cycle. Then  \[C(C_{2k}, H) =
 \frac{4k(k-1)}{2k\ell-2k-\ell}.\]
\end{theorem}

Next we prove that $C(H_1,H_2)$ where $H_1$ and $H_2$ are disjoint unions of
even cycles offers no surprises: all pure binomial inequalities involving
even cycles can be recovered from the density domination exponent
inequalities and concavity of cycles due to Theorem~\ref{thm:tropevencycles}
below.

\smallskip
 Recall that the {\em density profile} of a collection of connected graphs $\mathcal{U} = \{H_1, \ldots, H_s \}$, denoted by $\mathcal{D}_\mathcal{U}$,
 is the closure of the set of all vectors $(t(H_1,T), t(H_2,T), \ldots, t(H_s,T))$ as $T$ varies over all graphs (or graphons).
 In \cite{BRST}, it was shown that all valid pure binomial inequalities in $t(H_i,G)$'s are captured by
 the tropicalization of $\mathcal{D}_{\mathcal{U}}$, which as mentioned in the introduction is $ \textup{trop}(\mathcal{D}_{\mathcal{U}})  = \lim_{t \rightarrow 0}
 \log_{\frac{1}{t}}(\mathcal{D}_{\mathcal{U}})$ where $\log_b (\mathcal{D}_{\mathcal{U}})$ for some base $b$ denotes the image of $\mathcal{D}_{\mathcal{U}} \cap \RR^s_{>0}$ under the
map $\mathbf{v} \mapsto (\log_b v_1, \ldots, \log_b v_s)$.

\begin{theorem} \label{thm:tropevencycles}
Let $\mathcal{U}=\{C_2,C_4,C_6, \ldots, C_{2k}\}$, and let $\mathcal{D}_\mathcal{U}$ be the density graph profile of $\mathcal{U}$. Furthermore, let
\begin{align*}
Q_{\mathcal{U}}=\{(y_2, y_4, \ldots, y_{2k})\in \mathbb{R}^k\ |\ & y_{2i}-2y_{2i+2}+y_{2i+4} \geq 0 \ (1 \leq i \leq k-2)\\
&2ky_{2k-2} - (2k-2)y_{2k}\geq 0\\
& -(2k)y_2+y_{2k}\geq 0 \}.
\end{align*} Then
$\textup{trop}(\mathcal{D}_\mathcal{U}) = Q_{\mathcal{U}}$.
\end{theorem}

\begin{proof}
We first show that $\textup{trop}(\mathcal{D}_\mathcal{U}) \subseteq Q_\mathcal{U}$. We know that log-convexity inequalities $t(C_{2i},T)t(C_{2i+4},T)\geq t(C_{2i+2},T)^2$ are valid for $\mathcal{D}_\mathcal{U}$ for $i\geq 1$ (cf. the proof of \eqref{eq:concavity} above). By the same token, $t(C_{2k-2},T)^{2k} \geq t(C_{2k},T)^{2k-2}$. Finally, we know $t(C_{2k},T)\geq t(C_2,T)^{2k}$ from Sidorenko's inequality. Thus inclusion holds.

We now show that $\textup{trop}(\mathcal{D}_\mathcal{U}) \supseteq Q_\mathcal{U}$. We first claim that the extreme rays of $Q_\mathcal{U}$ are
$\mathbf{r}_i=(r_{i,2},r_{i,4}, r_{i,6}, \ldots, r_{i,2k})$ for $1 \leq i \leq k-1$ where $$r_{i,2j}=\left\{\begin{array}{ll}
-i-(j-1)(2i+1) & \textup{if } 1 \leq j \leq i+1, \textup{ and}\\
-i\cdot 2j & \textup{if } i+1 \leq j\leq k,\\
\end{array}\right.$$ and $\mathbf{s}=(-1,-2,-3,\ldots, -k)$. Indeed, for $1 \leq i \leq k-2$,
$\mathbf{r}_i$ is tight with all constraints except for
$y_{2i}-2y_{2i+2}+y_{2i+4}\geq 0$, $\mathbf{r}_{k-1}$ is tight with all
constraints except for $2ky_{2k-2} - (2k-2)y_{2k}\geq 0$, and $\mathbf{s}$ is
tight with all constraints except for $-(2k)y_2+y_{2k}\geq 0$. This shows
that the $(k\times k)$ matrix $M$ of constraints in the definition of
$Q_{\mathcal U}$ is invertible, with $\mathbf r_i, \mathbf s$ being, up to
positive normalizing factors, its rows. Thus, the cone $Q_{\mathcal U}$ is
obtained from the nonnegative orthant (i.e., the conical hull of the standard unit vectors $\mathbf{e}_i$) by an invertible linear transformation and
hence $\mathbf r_i, \mathbf s$ is the full set of its extreme rays.

To show the desired inclusion, we now need to show that these extreme rays are in $\textup{trop}(\mathcal{D}_\mathcal{U})$. To realize $\mathbf{r}_i$, let $T_{i,n}$ be a bipartite graph with parts of size $n^{i+1}$ and where each edge is present with probability $n^{-i}$ (and so we expect there to be $n^{i+2}$ edges). We have that $$\hom(C_{2j},T_{i,n})=\left\{ \begin{array}{ll}
O(n^{i+2+j-1}) & \textup{if } 1\leq j\leq i+1, \textup{ and}\\
O(n^{2j}) & \textup{if } i+1\leq j \leq k,
\end{array}\right.$$ and so $$t(C_{2j},T_{i,n})=\left\{ \begin{array}{ll}
O(n^{i+2+j-1-2j(i+1)}) & \textup{if } 1\leq j\leq i+1, \textup{ and}\\
O(n^{2j-2j(i+1)}) & \textup{if } i+1\leq j \leq k,
\end{array}\right.$$ which yields the desired ray by taking the log in the limit as $n\rightarrow \infty$. Finally, to realize $\mathbf{s}$, we simply take the graph on $n$ vertices containing exactly one edge as $n\rightarrow \infty$.
\end{proof}

\begin{remark}
Theorem \ref{thm:tropevencycles} gives an alternative proof of Theorem \ref{thm:even_cycles} for {\em even} $\ell$ (the lower bound in that theorem is realized by the ray $\mathbf r_{k-1}$, uniformly for all even $\ell\leq 2k$). In the opposite direction, our current inability to determine $C(C_{2k+1}, C_{2\ell+1})$ is the stumbling block for extending Theorem \ref{thm:tropevencycles} to odd cycles.
\end{remark}

\subsection{Bounds for $C(C_{2k+1}, C_{2\ell+1})$}
For odd cycles, the story is more complicated. For $2k+1\geq 2\ell+1$, Stoner showed that $\frac{k}{\ell} \leq C(C_{2k+1}, C_{2\ell+1}) \leq \lceil \frac{k}{\ell} \rceil +1$. We improve these bounds, and prove an asymptotically sharp bound.

We first need the following lemma that has been used in the literature many times, though often implicitly. It is sometimes called ``the tensor trick''.

\begin{lemma} \label{lem:relaxation}
Assume that for all finite graphs $T_n$ we have
\begin{equation} \label{eq:widetilde}
t(H_1,T_n) \geq \wt\Omega\of{t(H_2,T_n)^{c}},
\end{equation}
where the $\wt\Omega$-notation means constant multiplicative up to (poly)-logarithmic  factors in $n$. Then $C(H_1, H_2)\leq c$.
\end{lemma}
\begin{proof}
Let $T$ be any fixed graph, and consider its $r$-th tensor power
$\underbrace{T\times\ldots\times T}_{r\ \text{times}}$. Applying
\eqref{eq:widetilde} to it, we find that
$$
t(H_1,T)^r \geq \frac{t(H_2,T)^{cr}}{r^{O(1)}}.
$$
Letting $r\to\infty$, we get $t(H_1,T) \geq t(H_2, T)^c$, as desired.
\end{proof}

\begin{theorem} \label{thm:odd_cycles}
For $k>\ell$, we have
$$
C(C_{2k+1}, C_{2\ell+1}) \leq \begin{cases} \frac{2(k+1)}{2\ell+1} & \text{if } k\leq 2\ell-1, \textup{ and}\\
\frac{2k(2k-2\ell+1)-1}{2\ell(2k-2\ell+1)-1} & \text{if } k\geq 2\ell. \end{cases}
$$
\end{theorem}
\begin{proof}
Let us fix a graph $T$ with $n$ vertices and $m$ edges; we want to lower bound $t(C_{2k+1},T)$ in terms of $t(C_{2\ell+1},T)$. Construct recursively a sequence of target graphs $T=T_m,T_{m-1},\ldots, T_{m'} = T'$, where each is a spanning subgraph on $n$ vertices, where $T_{i-1}$ is obtained from $T_i$ as follows. If there is an edge that is contained in at most $\frac{\hom(C_{2\ell+1}, T_i)}{i\log n}$ homomorphism copies of $C_{2\ell+1}$, remove this edge. If there are no such edges, let $m':= i$ and stop the procedure.

Notice that $\sum_{i=1}^m\frac 1{i\log n}= O(1)$, which implies $\prod_{i=1}^m \of{1-\frac 1{i\log n}}=\Omega(1)$. In other words, a constant fraction of  homomorphism copies of $C_{2\ell+1}$ are still preserved in $T'$. Therefore, by Lemma \ref{lem:relaxation}, we can assume without loss of generality that $T'=T$. (Note that this idea is also often used in the literature and presented as pruning $T$ to a rainbow version by choosing a random balanced partition. However, because of the symmetry of $C_{2\ell+1}$, we prefer to give a simple ad hoc argument.)

Let $C_a^E$ be the labeled cycle on $a$ vertices in which two adjacent
vertices are labeled as $1$ and $2$. Given a graph $T$ with a labeled edge
$(u,v)$, we denote by $\hom_{u,v}(C_a^E, T)$ the number of homomorphisms from
$C_a$ to $T$ such that $(1,2)$ is mapped to the labeled edge $(u,v)$ in $T$.
In this language, what we have achieved so far is that
$$
\min_{\{u,v\}\in E(T)} \hom_{u,v}(C_{2\ell+1}^E, T) \geq \wt\Omega\of{\frac{\hom(C_{2\ell+1}, T)}{e(T)}}.
$$

Let $C_{2k+1}^+$ be the cycle $C_{2k+1}$ with an additional diagonal edge connecting two vertices such that this edge separates the cycle into two parts: one arc of $2\ell+1$ edges, and one arc of $2k-2\ell + 2$ edges.
We have
\begin{gather}
\begin{split}
 \label{eq:main}
  \hom(C_{2k+1}, T) \geq  & \\
 \hom(C_{2k+1}^+,T) = & \sum_{\{u,v\} \in E(T)}
\hom_{u,v}(C_{2\ell+1}^E,T)\hom_{u,v}(C_{2k - 2\ell+2}^E,T)
\\
\geq & \min_{\{u,v\}\in E(T)} \hom_{u,v}(C_{2\ell+1}^E, T)  \sum_{\{u,v\} \in E(T)}
\hom_{u,v}(C_{2k - 2\ell+2}^E,T)
\\
\geq & \wt\Omega\of{\frac{\hom(C_{2\ell+1}, T)}{e(T)}} \hom(C_{2k - 2\ell+2},T)
\\
= & \wt\Omega\of{\frac{\hom(C_{2\ell+1}, T) \hom(C_{2k-2\ell+2}, T)}{e(T)}}.
\end{split}
\end{gather}

Therefore we have (since \eqref{eq:main} is vertex-balanced)
\begin{equation} \label{eq:pruning}
t(C_{2k+1},T) \geq \wt\Omega\of{\frac{t(C_{2\ell+1}, T) t(C_{2k-2\ell+2}, T)}{t(K_2, T)}}.
\end{equation}
By Sidorenko's inequality, $t(K_2,T) \leq t(C_{2k-2\ell+2}, T)^{1/(2k-2\ell+2)}$.
Combining these two inequalities, we have
\[t(C_{2k+1},T) \geq \tilde \Omega \left(t(C_{2\ell+1},T)t(C_{2k-2\ell+2},T)^{1 - 1/(2k-2\ell+2)}\right).
\]
By  Theorem~\ref{thm:even_cycles} (applied with $k\mapsto k-\ell+1,\
\ell\mapsto 2\ell+1$), when $k\geq 2\ell$ we have
\[
t(C_{2k+1},T) \geq \tilde \Omega \of{t(C_{2\ell+1},T)t(C_{2\ell+1},T)^{(1 - 1/(2k-2\ell+2)) \cdot \frac{4(k-\ell+1)(k-\ell)}{4\ell(k-\ell+2)- (2\ell+1)}}}.
\]
Therefore,  when $k\geq 2\ell$, \[C(C_{2k+1}, C_{2\ell+1}) \leq (1 -
\frac{1}{(2k-2\ell+2)})\cdot \frac{4(k-\ell+1)(k-\ell)}{4\ell(k-\ell+1) -
(2\ell+1)} + 1.\] For the other case when $k\leq 2\ell-1$, we use
\cite[Proposition 4.3]{Stoner}:
\begin{eqnarray*}
t(C_{2k+1},G) &\geq & \tilde \Omega \of{t(C_{2\ell+1},T)t(C_{2k-2\ell+2},T)^{1 - 1/(2k-2\ell+2)}} \\
&\geq & \tilde\Omega\of{t(C_{2\ell+1},T)^{1 + (1 - \frac{1}{2k-2\ell+2}) \cdot \frac{2k-2\ell+2}{2\ell+1} }}.
\end{eqnarray*}
\end{proof}

The asymptotically matching lower bound is below.
\begin{theorem} \label{thm:odd_lower}
$$
C(C_{2k+1}, C_{2\ell+1}) \geq \frac{4k^2-1}{4k\ell-1}.
$$
\end{theorem}
\begin{proof}
The construction in the proof of Theorem \ref{thm:even_cycles} did not use
the fact that $2k$ is even: the only way it depends on $k$ is via the
definition \eqref{eq:alpha_def}. The upper bound on $\hom(C_{2k}, T_n)$ also
never used that $C_{2k}$ is even. Plugging $k\mapsto k+1/2$ in Theorem 4.1
gives the desired bound.
\end{proof}

\begin{corollary}
$$C(C_5,C_3)\in \left[ \frac{15}{7}, \frac{11}{5}\right].$$
Moreover, $C(C_5^+,C_3) = \frac{11}{5}$, where $C_5^+$ is $C_5$ with one added diagonal.
\end{corollary}
\begin{proof}
The statement for $C_5$ follows directly from Theorems \ref{thm:odd_cycles} and
\ref{thm:odd_lower}.

As for $C_5^+$, the only property of $C_{2k+1}$ we use in the proof of
Theorem \ref{thm:odd_cycles} is \eqref{eq:main}, and it holds for $C_5^+$ as
well. For the lower bound we again apply the construction from the proof of
Theorem \ref{thm:even_cycles} with $\alpha=2/3$ and {\em randomly} chosen set
of lines $L$. Due to this random choice, $\hom_{u,v}(C_3^E, T_n)$ and
$\hom_{u,v}(C_4^E, T_n)$ do not depend on the edge $(e,v)$, up to polylog
factors. Hence \eqref{eq:main} is tight, and we use it to {\em upper} bound
$t(C_5^+, T_n)$ in terms of $t(C_3, T_n), t(C_4,T_n)$.
\end{proof}


\section{Other computations and miscellaneous results}\label{sec:misc}

Let $\nu^\ast(H)$ be the fractional matching number of $H$ and recall that $P_m$ denotes the path with $m$ edges.

\begin{theorem} \label{thm:misc}
We have the following explicit values for $C(H_1,H_2)$.
\begin{enumerate}
\item  Suppose that $v(H_1)\leq v(H_2)$ and assume that every set
of $v(H_1)$ vertices in $H_2$ contains at least one copy of $H_1$ as a subgraph. Then
$C(H_1,H_2)=\frac{v(H_1)}{v(H_2)}$.
    \item We have that $C(P_2,H)= \frac{3}{v(H)}$ whenever $H$ can be
        vertex-covered by vertex-disjoint paths of length at least two. This means there exists a set of vertex disjoint paths $\{P_{\ell_1}, \cdots P_{\ell_k}\}$ each of length at least two such that $V(H)=\bigsqcup_{i=1}^k V(P_{\ell_i})$ and $E(H)\supseteq \bigsqcup_{i=1}^k E(P_{\ell_i})$.
    \item We have that $C(K_2, H) = \nu^\ast(H)^{-1}$. \label{thm:rho}

    \item If $H_1\subseteq H_2$ (not necessarily
        spanning) and $\nu^\ast(H_1)=\nu^\ast(H_2)$ then $C(H_1,H_2)$=1.
\end{enumerate}

\end{theorem}

\begin{proof}
\begin{enumerate}
\item This follows from the Kruskal-Katona theorem.

\item  From our assumption, we know that there exists a set of paths
    $\{P_{\ell_1}, \cdots P_{\ell_k}\}$, $\ell_i\geq 2$, such that
    $V(H)=\bigsqcup_{i=1}^k V(P_{\ell_i})$ and $E(H)\supseteq
    \bigsqcup_{i=1}^k E(P_{\ell_i})$. Then we have that

\begin{align*}
t(H,T) &\leq t(P_{\ell_1}\cup \cdots \cup P_{\ell_k},T)\\
&= \prod_{i=1}^k t(P_{\ell_i},T)\\
&\leq \prod_{i=1}^k t(P_2,T)^{\frac{\ell_i+1}{3}}\\
&=t(P_2,T)^{\frac{v(H)}{3}}
\end{align*}
where the penultimate line follows from the fact that we know that  $C(P_2,P_\ell)=\frac{3}{\ell+1}$ from Theorem~\ref{thm:Cpaths}. This yields that $C(P_2,T)\leq \frac{3}{v(H)}$.

The lower bound follows from the last part in Lemma \ref{lem:simple}.

\item This was already known from \cite{fracmatching}.

\item The upper bound follows from Lemma \ref{lem:basicproperties}(4), and
    the lower bound follows from Lemma \ref{lem:basicproperties}(3), with
    $F\mapsto K_2,\ G\mapsto H_1,\ H\mapsto H_2$, and Theorem
    \ref{thm:misc}(3).
\end{enumerate}

\end{proof}

\begin{corollary}
We have that $C(P_2,H)$ can be easily computed for every 4-vertex graph $H$.
\end{corollary}

\begin{theorem}\label{lem:K4-eK3}
 We have that  $C(K_4 -e, K_3) =2$.
\end{theorem}
\begin{proof}
It is known that $t(K_4 - e, W) \geq t(K_3, W)^2 \log^* t(K_3, W)$ (claimed
to be proved by Tao, see \cite[Exercise 16.21]{LL12}). This is a simple
Cauchy-Schwarz inequality together with the bound in the triangle removal
lemma. This inequality implies that $C(K_4 -e, K_3)\leq 2$.

To show that $C(K_4 -e, K_3)\geq 2$, we use the following construction.
 Let $N$ be a prime number and $G_N$ be the
Berhend graph, that is a graph on $3N$ vertices that is a union of at least
$N^{2-o(1)}$ edge-disjoint triangles, which does not contain any other
triangles. Then all homomorphism images of  $K_4-e$ are trivially the
original triangles, so $t(K_4, G_N)$ is of the order of $N^{-2}$. On the
other hand, the density $t(K_3, G_N)$ is of the order $N^{-1}$. This proves
the lower bound of $2$.
\end{proof}

\begin{theorem}
Let $H$ be a triangle with a pendant edge. Then
    $C(H, K_3) = 3/2$.
\end{theorem}
\begin{proof}
Let $T$ be any graph on $n$ vertices. Let $d_i$ be the number of triangles that
vertex $i$ is incident to. Then $\hom(K_3, T) \sim \sum_{i=1}^n d_i$. For
each vertex $i$, note that its degree $\deg_i$ should satisfy
$\binom{\deg_i}{2} \geq d_i$. Thus $\hom(H, T) \geq \Theta(\sum
d_i^{3/2})$, and so $t(H,T) \geq t(K_3,T)^{\frac{3}{2}}$ by H\"older
inequality.  On the other hand, by Lemma \ref{lem:simple}, the exponent
$\frac{3}{2}$ is sharp.
\end{proof}


\section{Open questions}\label{sec:open}

\subsection{General questions}

Some lower bound constructions $\{T_n\}$ used in our paper, such as
(quasi)-random graphs or cliques of fixed density, have the property that both 
$\lim_{n\to\infty}\log t(G,T_n)$ and $ \lim_{n\to\infty}\log t(H,T_n)$ exist and
are non-zero. For some of them, notably for the most sophisticated
construction in the proof of Theorem \ref{thm:even_cycles}, this is not true.

In the former case, the sequence $\{T_n\}$ can be replaced, after going to a
convergent subsequence, with a limiting graphon that leads to a much cleaner
version. Hence it is natural to ask the following question.

\begin{problem}\label{p2}
For which pairs $G, H$ with $t(G, H)>0$ does there exist a graphon $W$ with
$0<t(H, W)<1$, such that $t(G, W) = t(H, W)^{C(G,H)}$?
\end{problem}

The simplest example where such a graphon does not exist is $G=K_2,\ H=P_2$.
Indeed, $C(K_2,P_2)=1$, e.g., by Theorem \ref{thm:misc} (\ref{thm:rho}). On
the other hand, for a given $\rho\in [0,1]$, the maximum value of $t(P_2,W)$
taken over all graphons $W$ with $t(K_2,W)=\rho$ was computed in \cite{AK};
it can be easily checked that it is always strictly less than $\rho$ as long
as $\rho\in (0,1)$.

On the contrary, if $H$ contains a perfect matching then, again by Theorem
\ref{thm:misc} (\ref{thm:rho}), we have $C(K_2,H)=\frac{v(K_2)}{v(H)}$ and
hence the required graphon is simply the half-clique (which is the limit of the union of $K_n$ with $n$ disjoint vertices), as in the proof of Lemma
\ref{lem:simple}. We mildly conjecture that when $G=K_2$, this sufficient
condition for the existence of the realizing graphon $W$ is also necessary.

Baek and Lee (personal communication) have more sophisticated examples of
pairs $(G, H)$ that do not have a realizing graphon $W$. For more details
please refer to their upcoming work.

\medskip

We would also like to re-iterate the question already asked by Stoner \cite[Question 6.3]{Stoner}.

\begin{problem}
Is it true that for any $G,H$ with $t(G,H)>0$, $C(G,H)$ is rational? Algebraic?
\end{problem}

Finally, we pose the following problem.
\begin{problem}
Is the set of all valid pure binomial inequalities of the form
\eqref{eq:bound} (say, with rational $c_1,\ldots,c_k$) algorithmically
decidable?

Given the identity $t(G,T)\cdot t(H,T) = t(GH,T)$, it is easily seen to be
equivalent to asking whether the set $\{(G,H,c)\ |\ C(G,H)\leq c\}$ is
decidable. Another sensible form of this question is to ask the same when $G$
and $H$ are additionally required to be connected.
\end{problem}

\subsection{Some concrete homomorphism density domination exponents to compute}

We were able to determine $C(C_k,C_\ell)$ asymptotically, but the question of
computing it for fixed odd values when $k>\ell$ remains.

\begin{problem}
Find $C(C_k,C_\ell)$ exactly when $k$ and $\ell$ are odd, and $k>l$.
\end{problem}

In Theorem~\ref{thm:misc}, we computed $C(P_2,H)$ when $H$ can be vertex-covered by vertex-disjoint paths of length at least two, and we also know $C(K_2, H)$ for an arbitrary $H$. The Sidorenko conjecture is equivalent to computing $C(G, K_2)$, again for an arbitrary $G$. This opens up an interesting ``anti-Sidorenko'' direction of attempting to determine $C(G,H)$, where $G$ is fixed and ``very simple'' and $H$ is arbitrary. The following is the next natural step in that direction:

\begin{problem}
Find $C(P_2,H)$ for any graph $H$.
\end{problem}
Theorem \ref{thm:misc} and the work  in \cite{DR} cover some of the cases of $H$.

\subsection{Cycle inequalities and the tropicalization of the number graph profile of cycles}

When trying to compute the tropicalization for $\mathcal{D}_{\mathcal{U}}$
for $\mathcal{U}=\{C_2, C_3, \ldots, C_\ell\}$, the following inequalities
cropped up. We could not prove them in general nor find a counterexample.
Proving that they are valid or understanding what kind of constructions
breaks them might allow us to make some progress towards better understanding
$C(C_{2k+1},C_{2\ell+1})$.

\begin{problem}\label{lem:newineqevenoddodd}
Does the inequality
\[t(C_{2j},T)^2 t(C_{2i+1},T)^{2i-1-2j} \geq t(C_{2i-1},T)^{2i+1-2j}\] hold for all graphs $T$ when $i>j\geq 1$?

Note that this inequality is vertex-balanced, so it is equivalent to the same
inequality where we replace $t(\cdot)$ with $\hom(\cdot)$.
\end{problem}

We can show partial progress towards this conjecture, proving the case when $j=1, i=2$, as well as cases that would be consequences of the conjecture holding for other values. Following the notation of Kopparty and Rossman \cite{KR}, the homomorphism domination number exponent of $F$ and $G$, which is denoted as $\HDE(F, G)$, is defined as the maximal real number $c$ such that $\hom(F, T)\geq \hom(G, T)^c$ for all target graphs $T$.

\begin{lemma}
We have that $$t(C_{2},T)^{2i-2} t(C_{2i+1},T) \geq t(C_{3},T)^{2i-1}$$ for all $i\geq 2$.
\end{lemma}

\begin{proof}
We show instead that $\hom(C_{2},T)^{2i-2} \hom(C^c_{2i+1},T) \geq \hom(C_{3},T)^{2i-1}$ where $C^c_{2i+1}$ is the cycle with vertices $[2i+1]$ placed cyclically around the cycle and with chords $\{1,j\}$ for $j \in \{3,4, \ldots, 2i\}$. We let the vertices of $C_3$ be $\{1',2',3'\}$. Since all graphs in that inequality are chordal and series-parallel, Theorem 3.3 of \cite{KR} applies and we use their linear program to compute $\HDE(C_{2}^{2i-2} C^c_{2i+1},C_3)$:

\begin{align*}
\min \ & z\\
\textup{such that } & -p(1',3') + p(1') + p(3') \geq 0.0\\
& -p(1',2') + p(1') + p(2')\geq 0\\
&-p(2',3') + p(2') + p(3')\geq 0\\
& -p(1',2',3') + p(1',2') + p(1',3') -p(1')\geq 0\\
& -p(1',2',3') + p(1',2') + p(2',3') -p(2')\geq 0\\
&-p(1',2',3') + p(1',3') + p(2',3') -p(3')\geq 0\\
&p(1',2',3') = 1\\
&(2i-1) p(1',2',3') + (i-1) p(1',2') -(i-1) p(1',3') -z\leq 0\\
&(2i-1) p(1',2',3') -(i-1) p(1',2') + (i-1) p(1',3') -z\leq 0\\
&(2i-1)  p(1',2',3') -(i-1) p(1',2') -(i-1) p(1',3') + (2i-2) p(2',3') -z\leq 0\\
&(2i-1)  p(1',2',3') + (i-1) p(1',2') -(i-1) p(2',3') -z\leq 0\\
&(2i-1)  p(1',2',3') -(i-1) p(1',2') + (i-1) p(2',3') -z\leq 0\\
&(2i-1)  p(1',2',3') -(i-1) p(1',2') + (2i-2) p(1',3') -(i-1) p(2',3') -z\leq 0\\
&(2i-1)  p(1',2',3') + (i-1) p(1',3') -(i-1) p(2',3') -z\leq 0\\
&(2i-1)  p(1',2',3') -(i-1) p(1',3') + (i-1) p(2',3') -z\leq 0\\
&(2i-1)  p(1',2',3') + (2i-2) p(1',2') -(i-1) p(1',3') -(i-1) p(2',3') -z\leq 0.
\end{align*}

The first seven inequalities are the polymatroidal inequalities for $C_3$. The remaining inequalities are of the form $$\sum_{S\subseteq \textup{MaxCliques}(C_{2}^{2i-2}C_{2i+1}^c)} -(-1)^{|S|} p(\varphi(\cap S))$$ for every homomorphism $\varphi\in \Hom(C_{2}^{2i-2}C_{2i+1}^c,C_3)$ where $\textup{MaxCliques}(C_{2}^{2i-2}C_{2i+1}^c)$ is the set of maximal cliques of $C_{2}^{2i-2}C_{2i+1}^c$. It suffices to consider homomorphisms where the $2i-2$ copies of $C_2$ go to the same edge of $C_3$.  So for example, again letting the vertices of $C_{2i+1}^c$ be $[2j+1]$ as explained above and the vertices of $C_2$ being $\{2j+2,2j+3\}$, and those of $C_3$ being $\{1',2',3'\}$, the inequality $(2i-1) p(1,2,3) + (i-1) p(1,2) -(i-1) p(1,3) -z\leq 0$ comes from the homomorphism $\varphi$ such that $\varphi(1)=1'$, $\varphi(2j)=2'$ for $1 \leq j \leq i$, $\varphi(2j+1)=3'$ for $1 \leq j \leq i$, $\varphi(2j+2)=1'$ and $\varphi(2j+3)=2'$. The maximal cliques in $C_{2i+1}^c$ are the $2i-1$ triangles. In $\varphi$, all of those triangles go to the triangle $(1',2',3')$; this explains the term $(2i-1)p(1,2,3)$. There are $\binom{2i-1}{2}$ ways of choosing two cliques among the maximal cliques of $C_{2i+1}^c$. In $2i-2$ cases, we get the edges $\{1,j\}$ for $3\leq j \leq 2i$; $i-1$ of these edges go to edge $\{1',2'\}$ and $i-1$ go to edge $\{1',3'\}$, and so we add $-(i-1)p(1',2')-(i-1)p(1',3')$. In the remaining cases, we get vertex $1$. Note that vertex $1$ is also obtained as the intersection of any three or more maximal cliques of $C_{2i+1}^c$, and so $p(1)$ appears with coefficient $-(\binom{2i-1}{2}-(2i-2))+\binom{2i-1}{3}-\binom{2i-1}{4}+\ldots -(-1)^{2i-1}\binom{2i-1}{2i-1}=-\binom{2i-1}{0}+\binom{2i-1}{1}-\binom{2i-1}{2}+\binom{2i-1}{3}-\binom{2i-1}{4}+\ldots+\binom{2i-1}{2i-1}=0$. Finally, note that $C_2$ only has one maximal clique,  namely itself, for which all copies go to $\{1',2'\}$, and so we add $(2i-2)p(1',2')$. (Having the $C_2$ go to $\{1',3'\}$ or $\{2',3'\}$ would respectively yield the next two inequalities.) The other inequalities are obtained in similar fashion.

Note that
\begin{align*}
-(2i-1)\cdot & (p(1,2,3)=1)\\
\frac{1}{2} \cdot & ((2i-1)p(1,2,3)-(i-1)p(1,3)+(i-1)p(1,3)\leq z)\\
\frac{1}{2} \cdot & ((2i-1)p(1,2,3)+(i-1)p(1,3)-(i-1)p(1,3)\leq z)
\end{align*}
yields $z\geq 2i-1$, so we see that $\HDE(C_{2}^{2i-2} C^c_{2i+1},C_3)\geq 2i-1$. To show that $\HDE(C_{2}^{2i-2} C^c_{2i+1},C_3)=2i-1$, one can simply take a blow-up of $C_3$ where each vertex becomes a stable set with $n$ vertices, where edges are present with probability $n^{-1}$, and where triangles are present with probability $n^{-2}$.

Since there is a surjective homomorphism from $C_{2i+1}$ to $C^c_{2i+1}$, we have that  $$\hom(C_{2},T)^{2i-2} \hom(C_{2i+1},T)\geq \hom(C_{2},T)^{2i-2} \hom(C^c_{2i+1},T) \geq \hom(C_{3},T)^{2i-1},$$ and so the result holds.
\end{proof}

Should the answer to Problem ~\ref{lem:newineqevenoddodd} be yes, we could
improve the upper bound for $C(C_{2i+1},C_{2i-1})$. Indeed, using
\eqref{eq:pruning} (with $k\mapsto i, \ell\mapsto i-1$) and $t(C_4,T)\geq
t(K_2,T)^4$, we would obtain an upper bound of
              $$C(C_{2i+1}, C_{2i-1}) \leq \frac{2i - 1/3}{2i - 7/3}.$$

              Moreover, we claim we would get the following tropicalization for the number profile of all cycles, where the \emph{graph number profile} of $\mathcal{U}=\{C_2, C_3, C_4, C_5, \ldots C_{2m}\}$ is the set of vectors $$(\hom(C_2, T), \hom(C_3, T), \ldots, \hom(C_{2m},T))$$ as $T$ varies over all
              graphs.

Let $\mathcal{U}=\{C_2, C_3, C_4, C_5, \ldots C_{2m}\}$, let $\NU$ be the
graph homomorphism number profile of $\U$, and let $$Q_{\U}=\left\{ \begin{array}{lll}
(y_2,y_3,\ldots, y_{2m})\in\mathbb{R}^{2m-1} | & y_{2i-2}-2y_{2i}+y_{2i+2} \geq 0 & (2 \leq i \leq m-1)\\
& y_{2i}-2y_{2i+1}+y_{2i+2} \geq 0 &  (1 \leq i \leq m-1)\\
& 2y_{2i} - (2j+1-2i)y_{2i-1}+(2i-1-2i)y_{2j+1} \geq 0 &  (1 \leq i < j  < m)\\
& -y_{2i-1}+y_{2i+1} & (2 \leq i \leq m-1)\\
& y_3 \geq 0 & \\
&-y_{3}+y_4 \geq 0 & \\
& -y_{2} + y_{4} \geq 0 &\\
& m \cdot y_{2m-2} - (m-1) \cdot y_{2m} \geq 0 &
\end{array} \right\}.$$

It is easy to check that $Q_{\U}$ is the convex hull of the following rays:
$\mathbf{r}_{2i+1}=(r_2, r_3, \ldots, r_{2m})$ for $1\leq i \leq m$ and
$\mathbf{s}_{2i+1}=(s_2, s_3, \ldots, s_{2m})$ for $1 \leq i \leq m$ where

$$r_j=\left\{\begin{array}{ll}
0 & \textup{if } j \textup{ is odd and less than } 2i+1, \textup{ and}\\
j & \textup{otherwise,}
\end{array}\right.$$

and

$$s_j= \left\{\begin{array}{ll}
0 & \textup{if } j \textup{ is odd and less than } 2i+1, \textup{ and}\\
1& \textup{otherwise.}
\end{array}\right. $$

We skip an easy proof of the following equivalence.
\begin{lemma}
Problem \ref{lem:newineqevenoddodd} has an affirmative solution if and only
if $\trop(\NU) =Q_{\U}$.
\end{lemma}

\begin{remark}
Thus, if the answer to problem ~\ref{lem:newineqevenoddodd} is affirmative, then there are
no other interesting pure binomial inequalities in homomorphism numbers of
cycles (without using the vertex). For density inequalities, this means that
all remaining interesting inequalities to be found are not vertex-balanced
(since the vertex-balanced ones would give an inequality in cycle numbers
without using the vertex).
\end{remark}

\section*{Acknowledgments} We are grateful to Ingu Baek and Joonkyung Lee for
pointing out an issue with the original statement of Problem \ref{p2}.

\end{document}